\documentclass[12pt, reqno]{amsart}
\usepackage{BFHJY}

\title[A cubic PCF polynomial]{A large arboreal Galois representation
  for a cubic postcritically finite polynomial}

\author[Benedetto]{Robert L. Benedetto}
\address[Benedetto]{Amherst College \\ Amherst, MA}
\email[Benedetto]{rlbenedetto@amherst.edu}

\author[Faber]{Xander Faber}
\address[Faber]{Center for Computing Sciences \\
Institute for Defense Analyses \\ Bowie, MD}
\email[Faber]{awfaber@super.org}

\author[Hutz]{Benjamin Hutz}
\address[Hutz]{Saint Louis University \\
Saint Louis, MO}
\email[Hutz]{hutzba@slu.edu}

\author[Juul]{Jamie Juul}
\address[Juul]{Amherst College \\ Amherst, MA}
\email[Juul]{jamie.l.rahr@gmail.com}

\author[Yasufuku]{Yu Yasufuku}
\address[Yasufuku]{College of Science and Technology\\
Nihon University \\ Tokyo, Japan}
\email[Yasufuku]{yasufuku@math.cst.nihon-u.ac.jp}

\begin{document}
\begin{abstract}
We give a complete description of the arboreal Galois representation
of a certain postcritically finite cubic polynomial over a large class
of number fields and for a large class of basepoints. This is the
first such example that is not conjugate to a power map, Chebyshev
polynomial, or Latt\`es map. The associated Galois action on an
infinite ternary rooted tree has Hausdorff dimension bounded strictly
between that of the infinite wreath product of cyclic groups and that
of the infinite wreath product of symmetric groups. We deduce a
zero-density result for prime divisors in an orbit under this
polynomial. We also obtain a zero-density result for the set of places
of convergence of Newton's method for a certain cubic polynomial, thus
resolving the first nontrivial case of a conjecture of Faber and
Voloch.
\end{abstract}
\maketitle


\section{Introduction}

Let $K$ be a field and $f \in K[z]$ a polynomial of degree $d \geq 2$.
Consider the Galois groups of polynomials of the form
\begin{equation*}
	f^n(z) - x,
\end{equation*}
where $x \in K$, and $f^n = f \circ \cdots \circ f$ is the $n$-th
iterate of $f$ (with the convention that $f^0(z) = z$).
Such groups are called \emph{arboreal Galois groups} because (under
certain hypotheses) they can be made to act on trees.

Let $T_n$ be the graph whose vertex set is
\[
\bigsqcup_{0 \leq i \leq n} f^{-i}(x),
\]
and where we draw an edge from $\alpha$ to $\beta$ if $f(\alpha)
=\beta$.  Let $G_n = \Gal\left(f^n(z) - x / K\right)$. Clearly, $G_n$
acts faithfully on $T_n$, so that $G_n \hookrightarrow
\Aut(T_n)$. Provided there is no critical point of $f$ among the
points of the above vertex set, the graph $T_n$ is a regular $d$-ary
rooted tree with root $x$.  For such $f$, $\Aut(T_n)$ is isomorphic to
the $n$-fold iterated wreath product $[\fS_d]^n$ of the symmetric
group $\fS_d$ on $d$ letters.  Odoni and Juul
\cite{Juul_Galois_Iterates_2016,Odoni_Galois_Iterates_1985} showed
that if $\mathrm{char}(K)$ and the degree are not both $2$, and if $f$
is chosen generically (in the Zariski sense), then $G_n \cong
\Aut(T_n) \cong [\fS_d]^n$. 

By contrast, for \emph{specific} choices of $f$ and $x$, the
corresponding Galois groups may be much smaller.
(See \cite[\S3]{Jones_arboreal_survey_2013} for a high-level
explanation and
\cite{BHL_2016,Gottesman_Tang_Chebyshev,Jones_Manes_automorphisms,Pink_quadratic_polys}
for detailed examples.) Consider a polynomial $f$ that is
\emph{postcritically finite}, or \emph{PCF} for short, meaning that
all of its critical points have finite orbit under the iteration of
$f$. The simplest examples of PCF polynomials are the power maps $f(z)
= z^d$ and the Chebyshev polynomials defined by $f\left(z + 1/z\right)
= z^d + 1/z^d$. These two examples arise from the $d$-power
endomorphism of the algebraic group $\GG_m$, which gives a foothold on
the associated arboreal Galois representation. (A third type of
example, Latt\`es maps, arises from an endomorphism of an elliptic
curve; however, Latt\`es maps are never conjugate to polynomials. See
\cite[\S6.4]{Silverman_Dynamics_Book_2007}.)

Jones and Pink \cite[Thm.~3.1]{Jones_arboreal_survey_2013} have shown
that for PCF maps, the Galois groups $G_n$ have \emph{unbounded index}
inside $\Aut(T_n)$ as $n\to\infty$.  However, their proof does not
explicitly describe~$G_n$.
One can give an upper bound for~$G_n$ inside $\Aut(T_n)$ by realizing
it as a specialization of $\Gal\left(f^n(z) - t / K(t)\right)$, with
$t$ transcendental over $K$. The latter group may be embedded in the
profinite monodromy group $\pi_1^{\text{\'et}}(\PP^1_{K}
\smallsetminus P)$, where $P$ is the strict postcritical orbit; this
is precisely the tack taken by Pink \cite{Pink_quadratic_polys} in the
case of quadratic PCF polynomials.

In this paper, we give the first complete calculation of the arboreal
Galois group attached to a PCF polynomial over a number field that is
\emph{not} associated with an endomorphism of an algebraic group. More
specifically, we describe the Galois groups $G_n = \Gal\left(f^n(z) -
x / K\right)$ for the polynomial
\[
f(z) = -2z^3 + 3z^2
\]
over a number field $K$, where $x$ is chosen to satisfy a certain
local hypothesis at the primes $2$ and $3$.  In Section~\ref{Sec: tree
  automorphisms} we will define groups $E_n$ that fit between
$\Aut(T_n)\cong [\fS_3]^n$ and its Sylow 3-subgroup $[C_3]^n$ --- the
iterated wreath product of cyclic groups of order~3.  The groups $E_n$
are somewhat tricky to handle because their action on the tree lacks a
certain rigidity property:
for $m < n$, the kernel of the restriction homomorphism $E_n \to E_m$
is not the direct product of copies of $E_{n-m}$.  That is, in
contrast to $[\fS_3]^n$ and $[C_3]^n$, the action of $E_n$ on one
branch of the tree above $T_m$ is not independent of its action on
another branch.  Our main result is the following.

\begin{theorem}
\label{Thm:intro}
Let $K$ be a number field, let $f(z) = -2z^3 + 3z^2\in K[z]$, and let $x \in K$.
Suppose there exist primes $\pp$ and $\qq$ lying over $2$ and $3$, respectively,
such that $v_\qq(x) = 1$, and either $v_\pp(x) = \pm 1$ or $v_\pp(1-x) =
1$. Then for each $n \geq 1$,
\begin{enumerate}
\item The polynomial $f^n(z) - x$ is irreducible over $K$.
\item We have an isomorphism $\Gal\left(f^n(z) - x / K\right) \cong E_n \subset \Aut(T_n)$.
\end{enumerate}
Let $E_{\infty}=\varprojlim E_n$ and $\Aut(T_\infty) = \varprojlim \Aut(T_n)$ be the
corresponding inverse limits.
Then the Hausdorff dimension of $E_\infty$ in $\Aut(T_\infty)$ is
\begin{equation}
\label{Eq:HausDim}
\lim_{n \to \infty} \frac{\log |E_n|}{\log |\Aut(T_n)|}
= 1 - \frac{1}{3} \ \frac{\log 2}{\log 6} \approx 0.871.
\end{equation}
\end{theorem}

\begin{remark}
In this article, we implicitly work in the category of groups with an
action on the regular rooted tree $T_n$. This applies, for example, to
the isomorphism between $E_n$ and the Galois group in the theorem.
\end{remark}


The Galois group $\Gal\left(f^n(z) - x / K\right)$ depends
\textit{a priori} on the number field $K$ and the basepoint $x$, but
Theorem~\ref{Thm:intro} shows that many choices of $K$ and $x$ give the same
isomorphism type. One key reason is that the
discriminant of the second iterate is a square:
\begin{equation}
\label{Eq: discriminant}
\text{For any $x$, } \quad \mathrm{Disc}\left(f^2(z) - x\right) =
\left[2^{16} \cdot 3^9 \cdot x^2(x-1)^2\right]^2.
\end{equation}
This observation will be vital for forcing the Galois group of $f^n(z)
- x$ to lie inside $E_n$. To fill out the entire group $E_n$, we
utilize ramification above the primes~2 and~3. (See the proof of
Proposition~\ref{prop:E2}). These two features are the only
arithmetic-dynamical inputs to the theorem; the rest is general theory
of groups acting on regular rooted trees.

We are also able to deduce that the geometric Galois group has the same structure:

\begin{corollary}
\label{Cor:intro}
Let $f(z) = -2z^3 + 3z^2$. Let $t$ be transcendental over $\QQ$. For
every $n \geq 1$, we have
\[
\Gal\left(f^n(z) - t / \bar \QQ(t)\right) \cong E_n.
\]
\end{corollary}

For a polynomial $g \in \QQ[x]$, there are two profinite monodromy
groups:
\begin{align*}
  G_g^{\mathrm{geom}} &= \lim_{\substack{\leftarrow \\ n}} \Gal\left(g^n(z)
- t / \bar \QQ(t)\right) \qquad \text{(geometric monodromy)} \\
G_g^{\mathrm{arith}} &= \lim_{\substack{\leftarrow \\ n}} \Gal\left(g^n(z)
- t / \QQ(t)\right) \qquad \text{(arithmetic monodromy)}.
\end{align*}
In general, one knows that $G_g^{\mathrm{geom}} \subset
G_g^{\mathrm{arith}}$.  Theorem~\ref{Thm:intro} and its corollary
imply that $G_f^{\mathrm{geom}} = G_f^{\mathrm{arith}}$ for our
special cubic PCF polynomial $f(z) = -2z^3 + 3z^2$.  By contrast, Pink
has shown that $G_g^{\mathrm{geom}} \subsetneq G_g^{\mathrm{arith}}$
for all \textit{quadratic} PCF polynomials $g$ over the rationals
\cite[Thm.~2.8.4, Cor.~3.10.6]{Pink_quadratic_polys}. Similar
statements hold upon replacing $\QQ$ by essentially any other number
field.

While $E_n$ is not an iterated wreath product, it does satisfy the following
self-similarity property: the action of $E_n$ on the subtree of height $n-1$
stemming from any fixed vertex at level 1 is isomorphic to $E_{n-1}$.  This
self-similarity is a property of geometric iterated monodromy groups
\cite[Prop.~6.4.2]{Nekrashevych-self-similar}, and $E_n$ is such a group by
Corollary~\ref{Cor:intro}.

Odoni \cite{Odoni_Sylvester_1985} has shown that descriptions of
iterated Galois groups of this sort give rise to applications on the
density of prime divisors in certain dynamically defined
sequences. (See also
\cite{Jones_Quadratic_Polynomials_2008,Juul_Galois_Iterates_2016,Looper_Odoni_Conjecture}.)
More precisely, after counting elements of $E_n$ that fix a leaf of
the tree $T_n$, we have the following arithmetic application.

\begin{theorem}
\label{Thm:denseintro}
Let $K$ be a number field for which there exists an unramified prime
above $2$ and above $3$.  Let
$y_0 \in K \smallsetminus \{0,1,3/2,-1/2\}$,
and define the sequence $y_i = f^i(y_0)$. Then the set of
prime ideals $\fP$ such that
\[
y_i \equiv 0 \text{ or } 1 \!\!\!\!\pmod \fP \text{ for some $i \geq 0$}
\]
has natural density zero. In particular, the set of prime divisors of
the sequence $(y_i)$ has natural density zero.
\end{theorem}


Our choice of the polynomial $f(z) = -2z^3 + 3z^2$ was originally motivated by the following conjecture of Faber and Voloch \cite{Faber_Voloch_Bordeaux_2011}.

\begin{conjecture}[Newton Approximation Fails for 100\% of Primes]
\label{Conjecture: Newton conjecture}
	Let $g$ be a polynomial of degree $d \geq 2$ with coefficients
        in a number field $K$ and let $y_0 \in K$. Define the Newton
        map $N(z) = z - g(z) / g'(z)$ and, for each $n \geq 0 $, set
        $y_{n+1} = N(y_n)$. Assume the Newton approximation sequence
        $(y_n)$ is not eventually periodic. Let $C(K, g, y_0)$ be the
        set of primes $\fP$  of $K$ for which
        $(y_n)$ converges in the completion $K_{\fP}$ to a root of
        $f$. Then the natural density of the set $C(K, g, y_0)$ is
        zero.
\end{conjecture}

Faber and Voloch showed that Conjecture~\ref{Conjecture: Newton conjecture}
holds for any polynomial $g$ with at most 2 distinct roots. Thus, the first
nontrivial case of the conjecture is a separable cubic polynomial. For reasons
explained in \cite[Cor.~1.2]{Faber_Voloch_Bordeaux_2011}, the simplest such
cubic polynomial is $g(z) = z^3 - z$, whose associated Newton map turns out to
be conjugate to our polynomial $f(z) = -2z^3 + 3z^2$.  Our results therefore
yield a proof of the first nontrivial case of the Faber-Voloch conjecture:

\begin{theorem}
\label{Thm:FVintro}
Let $K$ be a number field for which there exists an unramified
prime over $2$ and over $3$. Let $g(z) = z^3 - z$. Choose $y_0 \in K$
such that the Newton iteration sequence $y_i = N^i(y_0)$ does not
encounter a root of $g$. Then the set of primes $\fP$ of $K$ for which
the Newton sequence $(y_i)$ converges in $K_\fP$ to a root of $g$ has
natural density zero.
\end{theorem}

The first and third authors, in collaboration with several others,
obtained a weak form of Theorem~\ref{Thm:FVintro} for a wide class of
polynomials \cite[Thm.~4.6]{BGHKST_periods}. More precisely, they
showed that the density of primes as in the theorem has natural
density strictly less than~1.

The outline of the paper is as follows.  In Section~\ref{Sec: tree
  automorphisms}, we will define and discuss the group $E_n$ and
compute the Hausdorff dimension of equation~\eqref{Eq:HausDim}.  We
will then prove the rest of Theorem~\ref{Thm:intro} in
Section~\ref{Sec: main theorem}.  Next, we consider the case that $K$
is the field of rational functions $\bar{\QQ}(t)$, proving
Corollary~\ref{Cor:intro} in Section~\ref{Sec: Geometric case}.  In
Section~\ref{Sec: Counting fixing elements}, we compute the proportion
of elements of $E_n$ that fix at least one leaf of the tree $T_n$, and
in Section~\ref{Sec: Applications}, we prove
Theorems~\ref{Thm:denseintro} and~\ref{Thm:FVintro}.



\section{Tree automorphisms}
\label{Sec: tree automorphisms}

Let $T_n$ denote the regular ternary rooted tree with $n$ levels, as
in Figure~\ref{fig:tree3}.
\begin{figure}
\includegraphics{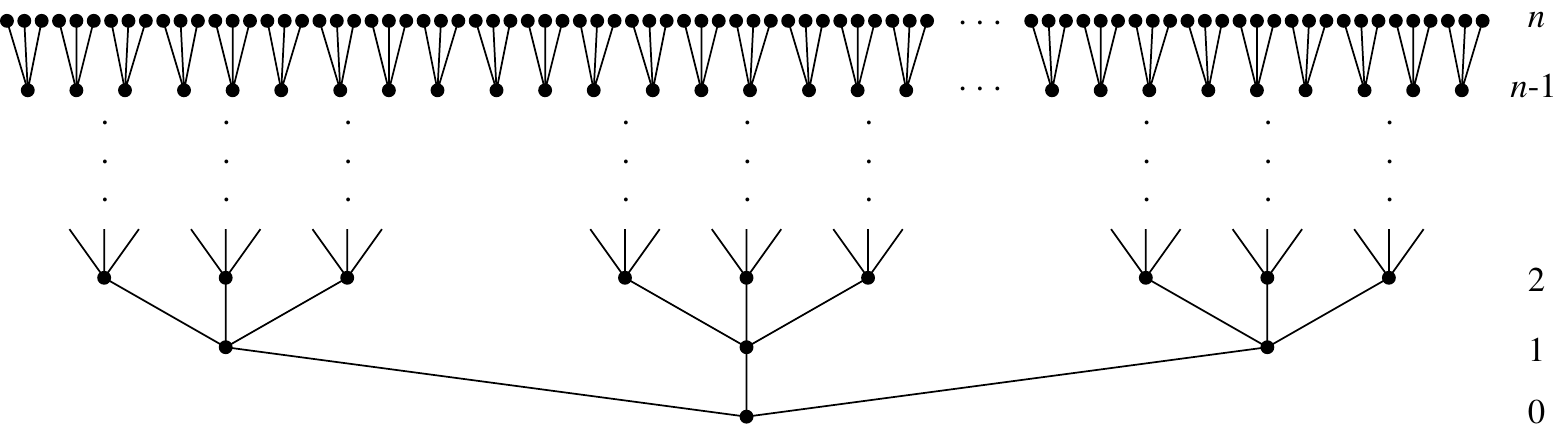}
\caption{The ternary rooted tree $T_n$ with $n$ levels}
\label{fig:tree3}
\end{figure}
Note that $T_n$ has $3^n$ leaves and $1 + 3 + \cdots + 3^n$ vertices.
Our results and many arguments will depend on an implicit labeling of
the vertices of $T_n$. We will make this labeling explicit now for
purposes of rigor, but we will not comment on it again afterward.
\begin{itemize}
  \item The \emph{level} of a vertex is its distance from the root.
  \item A vertex at level~$i$ is given a label $(\ell_1, \ldots,
    \ell_i)$, where $\ell_j \in \{1,2,3\}$. The root is given the
    empty label $()$.
  \item No two vertices at the same level have the same label.
  \item The unique path from the root to the vertex with label
    $(\ell_1, \ldots, \ell_i)$ consists of the vertices with labels
    $(), (\ell_1), (\ell_1,\ell_2), \ldots, (\ell_1, \ell_2, \ldots,
    \ell_i)$.
\end{itemize}
This labeling enables us to identify certain canonical subtrees of
$T_n$. For example, for each $i \in \{1,2,3\}$, we consider the subtree
$T$ that is induced by the set of vertices with labels of the form
$(i,*,\ldots, *)$; then $T$ is isomorphic to $T_{n-1}$.

The automorphism group $\Aut(T_n)$ of the regular rooted ternary tree is isomorphic
to the $n$-fold iterated wreath product $[\fS_3]^n$.
Indeed, we may decompose $T_n$ as a copy of $T_1$ (vertices of level at most~1)
with 3 copies of $T_{n-1}$ attached along the leaves of $T_1$, so that
\begin{equation}
\label{Eq:Autisom}
\Aut(T_n) \cong \Aut(T_{n-1}) \wr \Aut(T_1) \cong [\fS_3]^{n-1} \wr \fS_3 = [\fS_3]^n.
\end{equation}
Our labeling has the effect of fixing an isomorphism $\Aut(T_1) \cong
\fS_3$. For any elements $a_1,a_2,a_3 \in \Aut(T_{n-1})$ and any $b \in
\Aut(T_1)$, the element
\[
\big( (a_1,a_2,a_3),b\big) \in \Aut(T_{n-1}) \wr \Aut(T_1) \cong \Aut(T_n)
\]
acts on the tree by first acting on the 3 copies of $T_{n-1}$ via $a_1, a_2$,
and $a_3$, respectively, and then by permuting these $T_{n-1}$'s via $b$. More
precisely, if we label a vertex $y$ of $T_n$ by $(x,i)$, where $x$ is a vertex
of $T_{n-1}$ and $i \in \{1,2,3\}$ is the vertex at level 1 that lies below $y$,
then
\[
\big((a_1,a_2,a_3),b\big).(x,i) = (a_i.x, b.i).
\]

A labeling of the leaves of $T_n$ induces an injection $\Aut(T_n)
\hookrightarrow \fS_{3^n}$. Changing the labeling corresponds to conjugating by
an element of $\fS_{3^n}$, and so each automorphism of $T_n$ has a well-defined
sign attached to it, corresponding to the parity of the number of
transpositions needed to represent it. Thus, for each $n\geq 1$,
we have a homomorphism
\[
\sgn \colon \Aut(T_n) \to \{\pm 1\}.
\]

\begin{lemma}
\label{Lem: Sign}
Let $g = \big( (a_1,a_2,a_3),b\big)$ be an element of $\Aut(T_n)$ for some
$n \geq 2$, where $b \in \Aut(T_1)$, and $a_i \in \Aut(T_{n-1})$ for $i=1,2,3$.
Then
\[
\sgn(g) = \sgn(b) \  \prod_{i=1}^3 \sgn(a_i).
\]
\end{lemma}

\begin{proof}
Partition the leaves of $T_n$ into three disjoint sets $L_1, L_2, L_3$ so
that the elements of $L_i$ lie over leaf~$i$ of $T_1$, for $i=1,2,3$.
Note that $|L_i| = 3^{n-1}$.
With this notation, $\sgn(a_i)$ is the sign of $a_i$ acting
as a permutation on the set $L_i$.

Consider first the case that $b = 1$. Then $g$ permutes the elements of each
$L_i$ separately; hence, $\sgn(g) = \prod \sgn(a_i)$.

Next, consider the case that $g = \big( (1,1,1), b\big)$ for arbitrary
$b \in \Aut(T_1)$.  We have already proven the desired result if $b =
1$.  If $b$ is a 2-cycle --- say $b = (ij)$ --- then the induced
permutation on the leaves of $T_n$ decomposes as a product of
$3^{n-1}$ disjoint 2-cycles $(a_i a_j)$, where $a_i \in L_i$ and $a_j
\in L_j$. Therefore,
\[
\sgn(g) = (-1)^{3^{n-1}} = - 1 = \sgn(b).
\]
Similarly, if $b$ is a 3-cycle, then the induced permutation on the leaves of
$T_n$ decomposes as a product of $3^{n-1}$ disjoint 3-cycles. Hence,
\[
\sgn(g) = 1 = \sgn(b).
\]

Finally, we consider the general case $g = \big((a_1,a_2,a_3), b\big)$.
Define $h = \big((1,1,1), b^{-1}\big)$. Then $hg = \big((a_1,a_2,a_3),1\big)$.
The previous two paragraphs show that
\[
\prod_{i=1}^3\sgn(a_i) = \sgn(hg) = \sgn(h) \ \sgn(g) = \sgn(b^{-1}) \ \sgn(g).
\qedhere
\]
\end{proof}

For any $m \leq n$, we have a restriction homomorphism
$\pi_m \colon \Aut(T_n) \to \Aut(T_m)$, where $T_m$ is the subtree with $m$ levels
with the same root vertex as $T_n$.
We write $\sgn_m = \sgn \circ \pi_m$ for the composition of
restriction followed by the sign map.
Define a sequence of subgroups $E_n \subset \Aut(T_n)$ by the following formula:
\begin{equation*}
E_n =
\begin{cases}
\Aut(T_1) & \text{if $n = 1$} \\
\left(E_{n-1} \wr \Aut(T_1) \right) \cap \ker(\sgn_2) & \text{if $n \geq 2$}.
\end{cases}
\end{equation*}
Here we use the embedding
\[
E_{n-1} \wr \Aut(T_1) \hookrightarrow \Aut(T_{n-1}) \wr \Aut(T_1) \cong \Aut(T_n)
\]
from equation~\eqref{Eq:Autisom}. Thus, for $n\geq 2$, writing
a given automorphism $\sigma\in\Aut(T_n)$ as
$\sigma = \big((a_1,a_2,a_3), b \big)$, we have
\begin{equation}
\label{Eq:EnCrit}
\sigma = \big((a_1,a_2,a_3), b \big) \in E_n \qquad \text{if and only if}\qquad
a_1,a_2,a_3\in E_{n-1} \text{ and } \sgn_2(\sigma) = 1 .
\end{equation}

\begin{proposition}
\label{Prop: En order}
For $n \geq 1$, we have
$\displaystyle |E_n| = 2^{3^{n-1}} \cdot  3^\frac{3^n-1}{2}$.
\end{proposition}

\begin{proof}
Since $\Aut(T_1) \cong \fS_3$, the result is clear for $n = 1$. Suppose it holds
for some $n\geq 1$. Let $\phi$ be the composition
\[
E_{n} \wr \Aut(T_1) \hookrightarrow \Aut(T_{n+1}) \stackrel{\sgn_2}{\rightarrow} \{\pm 1\}.
\]
By definition, $E_{n+1} = \ker(\phi)$, and $\phi$ is onto because
$((1,1,1),\epsilon) \mapsto -1$ for any transposition $\epsilon$ of the leaves
of $T_1$. Thus,
\[
|E_{n+1}| = \frac{1}{2} \left| E_n \wr \Aut(T_1) \right|
= \frac{1}{2} |\Aut(T_1)| \cdot |E_n|^3
= \frac{1}{2} \cdot 6 \cdot \left( 2^{3^{n-1}} \cdot 3^{\frac{3^n-1}{2}} \right)^3
= 2^{3^n} \cdot 3^{\frac{3^{n+1}-1}{2}}.
\qedhere
\]
\end{proof}

Our construction of $E_n$ depends on an identification of $T_{n-1}$
with the subtree of $T_n$ lying above a vertex at level~1, which in
turn depends on the labeling we have assigned to $T_n$. In other words,
$E_n$ is not normal in $\Aut(T_n)$ (for $n \geq 3$): a different
labeling yields a conjugate subgroup in $\Aut(T_n)$.

\begin{proposition}
$E_n$ is normal in $\Aut(T_n)$ if and only if $n = 1$ or $2$.
\end{proposition}

\begin{proof}
We have $E_1 = \Aut(T_1)$, and $E_2$ has index~2 in $\Aut(T_2)$. It
remains to show that $E_n$ is not normal in $\Aut(T_n)$ for $n \geq 3$. To that
end, we first construct some special elements of $\Aut(T_n)$.

Define $\nu_n \in \Aut(T_n)$ inductively for $n\geq 1$ as follows:
\[ \nu_n =
\begin{cases}
    (12) & n = 1 \\
    \big((\nu_{n-1},1,1), 1 \big) & n \geq 2.
\end{cases}
\]
Thus, $\nu_n$ transposes two leaves at the $n$-th level and acts
trivially on the rest of $T_n$.  In particular, $\sgn(\nu_n) =
-1$. This yields $\nu_2 \not\in E_2$, and by induction, it follows
that $\nu_n \not\in E_n$ for $n \geq 2$. Note further that $\nu_n^{-1}
= \nu_n$.

Next, for fixed $n\geq 3$, define $a=\big( (1,1,1), (123) \big)\in\Aut(T_n)$.
Then $a\in E_n$  by~\eqref{Eq:EnCrit}. However,
\[ \nu_n \, a \, \nu_n^{-1} =
\big( (\nu_{n-1},1,1),1 \big) \big( (1,1,1),(123) \big) \big(
(\nu_{n-1},1,1),1 \big) = \big( (\nu_{n-1},1,\nu_{n-1}),(123)
\big), \] which does \emph{not} belong to $E_n$ since $\nu_{n-1}
\not\in E_{n-1}$ for $n \geq 3$. It follows that $E_n$ is not normal
in $\Aut(T_n)$, as desired.
\end{proof}

Write $T_\infty = \bigcup_{n \geq 1} T_n$ for the infinite ternary rooted tree,
which has automorphism group
\[\Aut(T_\infty) = \varprojlim \Aut(T_n), \]
where the inverse limit is taken with respect to the restriction homomorphisms
\[ \pi_m:\Aut(T_n) \to \Aut(T_m) \quad\text{for}\quad m \leq n. \]
The recursive definition of $E_n$
implies that we also have restriction homomorphisms $E_n \to E_m$ for $m \leq n$.
Passing to the inverse limit gives a subgroup
\[ E_\infty = \varprojlim E_n \]
of $\Aut(T_\infty)$.

\begin{corollary}
  \label{Cor: Hausdorff}
The Hausdorff dimension of $E_\infty$ in $\Aut(T_\infty)$ is
given by equation~\eqref{Eq:HausDim}.
\end{corollary}

\begin{proof}
Using the facts that $\Aut(T_n) \cong \Aut(T_{n-1}) \wr \Aut(T_1)$
and that $\Aut(T_1)\cong \fS_3$, a simple induction shows that
\begin{equation}
\label{Eq:AutTnSize}
|\Aut(T_n)| = 6^{\frac{3^n - 1}{2}}.
\end{equation}
Combining this fact with Proposition~\ref{Prop: En order}
gives the desired result.
\end{proof}

\begin{corollary}
$E_\infty$ has infinite index in $\Aut(T_\infty)$.
\end{corollary}

Comparing the cardinalities of $E_n$ and $\Aut(T_n)$, we see that they share a
Sylow 3-subgroup. We can describe one such subgroup explicitly as follows.
Let $C_3$ be the cyclic 3-subgroup of $\Aut(T_1) \cong \fS_3$.
Define a sequence of groups $H_n$ by the following formula:
\[
H_n =
\begin{cases}
C_3 & \text{if $n = 1$} \\
H_{n-1} \wr C_3 & \text{if $n \geq 2$}.
\end{cases}
\]
We identify $H_n$ with a subgroup of $\Aut(T_n)$ using the embedding
\[
H_n = H_{n-1} \wr C_3 \hookrightarrow \Aut(T_{n-1}) \wr \Aut(T_1)
\cong \Aut(T_n).
\]
Evidently, $H_n \cong [C_3]^n$, the iterated wreath product. By
induction, we see that
\begin{equation}
\label{HnSize}
|H_n| = 3^{\frac{3^n - 1}{2}},
\end{equation}
so that $H_n$ is a Sylow 3-subgroup of $\Aut(T_n)$.

\begin{proposition}
  \label{Prop: sylow}
For $n \geq 1$, $H_n$ is a Sylow 3-subgroup of $E_n$. It is normal in $E_n$
if and only if $n=1$.
\end{proposition}

\begin{proof}
For $n = 1$, $H_n$ is an index-2 subgroup of $E_1 = \Aut(T_1)$ and, hence, it is normal.
Next, for some $n\geq 1$, suppose we know that
$H_n$ is a subgroup of $E_n$. By their
recursive definitions, to see that $H_{n+1} \subset E_{n+1}$ it suffices to
show that any $h \in H_{n+1}$ restricts to an even permutation on $T_2$. The
restriction of $h$ to $T_2$ has the form $\big((a_1,a_2,a_3), b\big)$, where
each of $a_1, a_2, a_3$, and $b$ is either trivial or a 3-cycle. By
Lemma~\ref{Lem: Sign}, we conclude that $\sgn_2(h) = 1$. Hence, $h \in E_{n+1}$, as desired.

Since $|H_n|$ is the 3-power part of $|E_n|$, we have proven the first statement of the proposition.
It remains to show that $H_n$ is not normal in $E_n$ for $n \geq 2$.

For each $n\geq 1$, define $\tau_n \in \Aut(T_n)$ inductively as follows:
\begin{equation}
\label{Eq: tau}
\tau_n =
\begin{cases}
	(12) & n = 1 \\
	\big((\tau_{n-1},1,1), (12) \big) & n \geq 2.
\end{cases}
\end{equation}
We claim that $\tau_n \in E_n$ for all $n\geq 1$. This is clear for $n=1$.
Suppose that it holds for some $n\geq 1$. Then $\big((\tau_n, 1, 1), (12)\big) \in E_{n+1}$
if and only if its restriction to $T_2$ acts by an even permutation. The
restriction to $E_2$ is given by $\big(((12), 1, 1), (12)\big)$, which has positive sign
by Lemma~\ref{Lem: Sign}.

Note that for $n \geq 2$, we have
$\tau_n^{-1} = \big((1, \tau_{n-1}^{-1}, 1), (12)\big)$.
Note also that for $n\geq 1$, we have $\tau_n\not\in H_n$,
since the restriction of $\tau_n$ to $T_1$ is $(12)\not\in C_3$.


Next, for fixed $n\geq 2$, define $a= \big( (1,1,1), (123) \big)\in H_n$. Then
\[
\tau_n \, a \, \tau_n^{-1} =
\big( (\tau_{n-1},1,1),(12) \big) \big( (1,1,1),(123) \big)
\big((1,\tau_{n-1}^{-1},1),(12) \big)
= \big((1,\tau_{n-1}^{-1},\tau_{n-1}),(132) \big),
\]
which does not belong
to $H_n$, since $\tau_{n-1}\not\in H_{n-1}$.
\end{proof}

%
%

\begin{proposition}
The Hausdorff dimension of $H_\infty$ in $\Aut(T_\infty)$ is
\[ \lim_{n \to \infty} \frac{\log |H_n|}{\log |\Aut(T_n)|}
= \frac{\log 3}{\log 6} \approx 0.613. \]
\end{proposition}

\begin{proof}
Immediate from equations~\eqref{Eq:AutTnSize} and~\eqref{HnSize}.
\end{proof}

Since $H_n \subseteq E_n \subseteq \Aut(T_n)$, the preceding
proposition and Corollary~\ref{Cor: Hausdorff} show that for large
$n$, $E_n$ is substantially larger than $H_n$, but much smaller than
$\Aut(T_n)$.

Finally, we will need a lemma that constructs certain special elements of $E_n$:

\begin{lemma}
  \label{Lem: Special elts}
  Let $n \geq 2$ and let $g \in \Aut(T_n)$ be any element that acts as follows:
  \begin{itemize}
    \item On the copy of $T_{n-1}$ with the same root as $T_n$, $g$
      acts by the identity.
    \item On each copy of $T_2$ rooted at a vertex of $T_n$ of level
      $n-2$, $g$ acts by an even permutation of the $9$ leaves.
  \end{itemize}
  Then $g \in E_n$.
\end{lemma}

\begin{proof}
  We proceed by induction on $n$. For $n=2$, the second condition on
  $g$ implies that $\sgn(g) = 1$, so $g \in E_2$. Suppose that the
  lemma holds for $n-1$, and let $g \in \Aut(T_n)$ satisfy the given
  conditions. Let $u_1,u_2,u_3$ be the vertices of $T_n$ at
  level~1. Write $T_{n-1}(u_i)$ for the copy of $T_{n-1}$ inside $T_n$
  that is rooted at $u_i$. Then $g$ restricts to an element of
  $\Aut(T_{n-1}(u_i))$ that satisfies the two conditions of the
  lemma. By the induction hypothesis, $g|_{T_{n-1}(u_i)} \in
  E_{n-1}$ for $i = 1,2,3$. In addition, $g$ is the identity and, hence, is even,
  on $T_2$. Thus, $g\in E_n$ by the criterion of~\eqref{Eq:EnCrit}.
\end{proof}

\section{Main Theorem}
\label{Sec: main theorem}

Let $K$ be a field of characteristic zero, and consider the polynomial
\[
f(z) = -2z^3 + 3z^2 \in K[z].
\]
The critical points of $f$ in $\PP^1$ are $0$, $1$, and $\infty$, all
of which are fixed by $f$. Hence, $f$ is PCF and, in fact, the union
of the forward orbits of its critical points is $\{0,1,\infty\}$.
Choose a point $x\in K\smallsetminus\{0,1\}$ to be the root of our
preimage tree. Note that there is no critical point in the backward
orbit of $x$ --- i.e., the set $f^{-1}(x) \cup f^{-2}(x) \cup \cdots$.

For each $n \geq 1$, define
\begin{equation}
\label{Eq:KnGnDef}
K_n=K\big(f^{-n}(x)\big) \qquad\text{and}\qquad G_n = \Gal(K_n / K).
\end{equation}

\begin{lemma}\label{Lem:GnsubEn}
For any field $K$ of characteristic zero and any $x\in K\smallsetminus\{0,1\}$,
the Galois group $G_n$ of~\eqref{Eq:KnGnDef} is isomorphic to a subgroup of $E_n$.
\end{lemma}

\begin{proof}
Because there is no critical point in the backward orbit of $x$,
the set $f^{-i}(x)$ consists of exactly $3^i$ distinct elements for each $i\geq 0$.
Identify the vertices of the ternary rooted tree $T_n$ with the set
$\bigsqcup_{0 \leq i \leq n} f^{-i}(x)$, with vertex $y$ lying immediately above $y'$
if and only if $f(y)=y'$.
This identification induces a faithful action of $G_n$ on $T_n$
and, hence, $G_n$ may be identified with a subgroup of $\Aut(T_n)$.

To see that this subgroup $G_n$ lies inside $E_n$, we proceed
by induction on $n$. For $n=1$, this is clear because $E_1=\Aut(T_1)$.

Fix  $n\geq 2$, and assume we know the lemma holds for $n-1$.
Write $f^{-1}(x)=\{y_1,y_2,y_3\}$. For each $i=1,2,3$, applying the lemma to
the field $K(y_i)$ with root point $y_i$ shows that
\[ \Gal\Big(K\big(f^{-(n-1)}(y_i)\big)/K(y_i)\Big) \]
is a subgroup of $E_{n-1}$. (The labeling of $T_n$ allows us to
identify the portion of $T_n$ above $y_i$ with $T_{n-1}$.)  Since
$\Gal(K_1/K)$ is a subgroup of $\Aut(T_1)$, it follows
that $G_n$ is isomorphic to a subgroup $B_n$ of $E_{n-1} \wr
\Aut(T_1)$.

It remains to show that $B_n\subseteq\ker(\sgn_2)$.
Direct computation shows that the discriminant of the degree-nine polynomial
$f^2(z) - x$ is given by equation~\eqref{Eq: discriminant}. Since
this discriminant is a square in $K$,
all elements of $B_n$ act as even permutations of the nine points of $f^{-2}(x)$.
Thus, $B_n\subseteq\ker(\sgn_2)$, as desired.
\end{proof}

Our goal is to compute the arboreal Galois groups $G_n$ in the
case that $K$ is a number field and that the basepoint  $x \in K\smallsetminus\{0,1\}$
satisfies the following local hypothesis:
\begin{align*}
\label{Dagger property}
&\text{There exist primes $\pp$ and $\qq$ of $K$ lying above $2$ and $3$,
respectively,} \tag{$\dagger$} \\
&\text{such that $v_{\qq}(x)  = 1$, and either $v_{\pp}(x) = \pm 1$ or
$v_{\pp}(1-x) = 1$.}
\end{align*}
If this hypothesis holds, we will say that the pair $(K,x)$ satisfies property
\eqref{Dagger property} (relative to $\pp$ and $\qq$).

\begin{example}
If $K = \QQ$, then the pairs $(\QQ,3)$ and $(\QQ,3/2)$
both satisfy property \eqref{Dagger property}.
The latter pair will be important for our arithmetic applications.
\end{example}

\begin{lemma}
\label{Lem: eisenstein}
Suppose that $(K,x)$ satisfies \eqref{Dagger property} relative to $\pp$ and
$\qq$. Then $f^n(z) - x$ is Eisenstein at $\qq$ for all $n \geq 1$. In particular,
$f^n(z) - x$ is irreducible for all $n \geq 1$.
\end{lemma}

\begin{proof}
  A simple induction shows that $f^n(z) \equiv z^{3^n} \pmod \qq$ and
  $f^n(0) = 0$. Since $v_{\qq}(x) = 1$, it follows immediately that
  $f^n(z) - x$ is Eisenstein at $\qq$.
\end{proof}

%
%

\begin{proposition}
\label{prop:div6n}
Let $K$ be a number field, and let $x \in K$.
Suppose that $(K,x)$ satisfies
property \eqref{Dagger property} relative to primes $\pp$ and $\qq$.
Let $n\geq 0$, and let $y\in f^{-n}(x)$.
Then:
\begin{enumerate}
	\item There are primes $\pp'$ and $\qq'$ of $K(y)$ lying above $\pp$ and $\qq$,
	respectively, such that
	\[ e(\pp'/\pp)  = 2^n \quad\text{and}\quad e(\qq' / \qq) = 3^n.\]
	\item The pair $(K(y),y)$ satisfies property \eqref{Dagger property} relative
	to $\pp'$ and $\qq'$.
\end{enumerate}
\end{proposition}

\begin{proof}
We proceed by induction on $n$. The statement is trivial for $n=0$. We therefore
assume for the rest of the proof that $n\geq 1$ and that the statement holds for $n-1$.

Given $y\in f^{-n}(x)$, let $y''=f(y)\in f^{-(n-1)}(x)$. By our inductive hypothesis,
there are primes $\pp''$ and $\qq''$ of $K(y'')$ lying over $\pp$ and $\qq$
satisfying the desired properties for $n-1$. The polynomial
\[ f(z)-y'' = -2z^3 + 3z^2 - y'' \in K(y'')[z] \]
is Eisenstein at $\qq''$. Thus, there is only one prime $\qq'$ of $K(y)$ lying above $\qq''$,
with ramification index $e(\qq' / \qq'') = 3$, and with $v_{\qq'}(y) = 1$. Moreover,
\[ e(\qq' / \qq) = e(\qq' / \qq'') \cdot e(\qq'' / \qq) = 3 \cdot 3^{n-1} = 3^n .\]
Meanwhile, by statement~(2) for $n-1$, we have either $v_{\pp''}(y'')=1$,
$v_{\pp''}(1-y'')=1$, or $v_{\pp''}(y'')=-1$. We consider these three cases separately.

If $v_{\pp''}(y'')=1$, then the Newton polygon of $f(z)-y''$ at $\pp''$ has a segment of length $2$
and height $1$. Thus, there is a prime $\pp'$ of $K(y)$ lying above $\pp''$,
with ramification index $e(\pp' / \pp'') = 2$, and with $v_{\pp'}(y) = 1$. Moreover,
\[ e(\pp' / \pp) = e(\pp' / \pp'') \cdot e(\pp'' / \pp) = 2 \cdot 2^{n-1} = 2^n .\]

If $v_{\pp''}(1-y'')=1$, note that $f$ is self-conjugate via $z\mapsto 1-z$; that is,
$1-f(1-z) = f(z)$. Thus, $1-y\in f^{-1}(1-y'')$, and the previous paragraph applied to $1-y$
gives the desired conclusion.

Finally, if $v_{\pp''}(y'')=-1$, then because $v_{\pp''}(-2)\geq 1$,
the Newton polygon of $f(z)-y''$ at $\pp''$ has a segment of length $2$
and height $-1$. Thus, there is a prime $\pp'$ of $K(y)$ lying above $\pp''$,
with ramification index $e(\pp' / \pp'') = 2$, and with $v_{\pp'}(y) = -1$.
Once again, then, we have $e(\pp'/\pp)=2^n$.
\end{proof}

\begin{corollary}
\label{cor:6n}
Let $K$ and $x$ be as in Proposition~\ref{prop:div6n}.
Let $n\geq 1$ and let $K_n$ be the splitting field of $f^n(z)-x$ over $K$. Then
\[ 6^n \big| [K_n : K] . \]
\end{corollary}

\begin{proof}
Pick $y\in f^{-n}(x)$, and let $\pp'$ and $\qq'$ be the primes of
$K(y)$ given by Proposition~\ref{prop:div6n}.  Since $K_n/K$ has
intermediate extension $K(y)/K$, the ramification index of some prime
of $K_n$ over $\pp$ must be divisible by $e(\pp'/\pp)=2^n$.
Similarly, the ramification index of some prime of $K_n$ over $\qq$
must be divisible by $e(\qq'/\qq)=3^n$. Thus, $6^n \mid [K_n:K]$.
\end{proof}

\begin{proposition}
\label{prop:E2}
Let $K$ and $x$ be as in Proposition~\ref{prop:div6n}. Then
\[
\Gal\Big(K\big(f^{-1}(x)\big) / K\Big) \cong E_1 \cong \fS_3
\quad \text{and} \quad
\Gal\Big( K\big(f^{-2}(x) \big) / K \Big) \cong E_2 .
\]
\end{proposition}

\begin{proof}
Let $K_1=K(f^{-1}(x))$ and $K_2= K(f^{-2}(x))$.  Then, as a
splitting field of a cubic polynomial, $K_1$ is Galois over $K$ with
$\Gal(K_1/K)$ isomorphic to a subgroup of $\fS_3$.  By
Corollary~\ref{cor:6n} with $n=1$, we have $6\mid[K_1:K]$. Hence,
$\Gal(K_1/K)\cong \fS_3$.

By Lemma~\ref{Lem:GnsubEn},
$\Gal(K_2/K)$ acts on $f^{-2}(x)$ as a subgroup of $E_2$.
It suffices to show that every element of $E_2$ is realized in $\Gal(K_2/K)$.

Write $f^{-1}(x) = \{ u_1, u_2, u_3 \}$, and for each $i=1,2,3$, write
\[ f^{-1}(u_i) = \{v_{i1}, v_{i2}, v_{i3} \} . \]

\noindent
\textbf{Claim 1.} There exists $\tau\in\Gal(K_2/K_1) \subseteq \Gal(K_2/K)$ that
\begin{itemize}
\item fixes $v_{1j}$ for each $j=1,2,3$,
\item acts as a 2-cycle on the set $f^{-1}(u_2) = \{v_{21}, v_{22},v_{23} \}$, and
\item acts as a 2-cycle on the set $f^{-1}(u_3) = \{v_{31}, v_{32},v_{33} \}$.
\end{itemize}

To prove Claim~1, note that $36 \mid [K_2:K]$ by
Corollary~\ref{cor:6n} and, hence, $6 \mid [K_2 : K_1]$, since
$[K_1:K]=6$. By Cauchy's Theorem, there is some
$\tau_1\in\Gal(K_2/K_1)$ of order $2$.  Since $\tau_1$ fixes each
$u_i$, it must act on each set $f^{-1}(u_i)$ as an element of $\fS_3$ of
order dividing $2$. Thus, for each $i=1,2,3$, $\tau_1$ acts as either
a 2-cycle or the identity on $f^{-1}(u_i)$.  In addition, $\tau_1$ is
even as a permutation on $f^{-2}(x)$ but (being of order~2) is not the
identity. Thus, $\tau_1$ acts as a 2-cycle on the preimages of exactly
two of $u_1,u_2,u_3$, and as the identity on the preimage of the
third, $u_m$.

Choose $\gamma' \in \Gal(K_1/K)$ with $\gamma'(u_1)=u_m$, and lift $\gamma'$ to
$\gamma\in\Gal(K_2/K)$. Let $\tau = \gamma^{-1}\tau_1\gamma$. Then $\tau(u_i)=u_i$
for each $i$, and $\tau$ satisfies each of the three bulleted properties,
proving Claim~1.

\medskip

\noindent
\textbf{Claim 2.} There exists $\rho\in\Gal(K_2/K_1) \subseteq \Gal(K_2/K)$
that acts as
\begin{itemize}
\item a 3-cycle on $f^{-1}(u_1) = \{v_{11}, v_{12},v_{13} \}$,
\item either a 3-cycle or the identity on $f^{-1}(u_2) = \{v_{21}, v_{22},v_{23} \}$, and
\item either a 3-cycle or the identity on $f^{-1}(u_3) = \{v_{31}, v_{32},v_{33} \}$.
\end{itemize}

To prove Claim~2, we again note that $6 \mid [K_2 : K_1]$. Hence,
by Cauchy's Theorem, there is some $\rho_1\in\Gal(K_2/K_1)$ of order
$3$. We have $\rho_1(u_i)=u_i$ for each $i$, with $\rho_1$ acting as a
3-cycle on either one, two, or all three of $f^{-1}(u_i)$, and as the
identity on the others. Lifting some appropriate
$\gamma'\in\Gal(K_1/K)$ to $\gamma\in\Gal(K_2/K)$, we can ensure that
$\rho = \gamma^{-1}\rho_1\gamma$ satisfies the desired properties,
proving Claim~2.

\medskip

By the conclusions of Claims~1 and~2, the permutation $\tau\rho$ acts as
\begin{itemize}
\item a 3-cycle on $f^{-1}(u_1) = \{v_{11}, v_{12},v_{13} \}$,
\item a 2-cycle on $f^{-1}(u_2) = \{v_{21}, v_{22},v_{23} \}$, and
\item  a 2-cycle on  $f^{-1}(u_3) = \{v_{31}, v_{32},v_{33} \}$.
\end{itemize}
Let $\sigma=(\tau\rho)^2$. Then $\sigma$ acts as a $3$-cycle on $f^{-1}(u_1)$
and as the identity on $f^{-1}(u_2) \cup f^{-1}(u_3)$.

Conjugating $\sigma$ by
permutations $\gamma$ as in the proof of Claim~1 and then composing the resulting
conjugates with one another, we see that $\Gal(K_2/K_1)$ contains
each element of $E_2$ that is the identity on $f^{-1}(x)$ and is either the identity
or a $3$-cycle on each $f^{-1}(u_i)$. There are $3^3=27$ such permutations.

Conjugating $\tau$ by permutations from the previous paragraph, as
well as by permutations $\gamma$ as in the proof of Claim~1, we also
see that $\Gal(K_2/K_1)$ contains each element of $E_2$ that is the
identity on $f^{-1}(x)$, is a $2$-cycle on exactly two of
$f^{-1}(u_1)$, $f^{-1}(u_2)$, and $f^{-1}(u_3)$, and is the identity
on the third. There are $3^3=27$ such permutations.

Composing the maps of the previous two paragraphs, we see that
$\Gal(K_2/K_1)$ contains each element of $E_2$ that is the identity on $f^{-1}(x)$,
is a $2$-cycle on exactly two of $f^{-1}(u_1)$, $f^{-1}(u_2)$, and $f^{-1}(u_3)$, and
is a $3$-cycle on the third. There are $3^3 \cdot 2 = 54$ such permutations.

Thus, $[K_2 : K_1] \geq 27 + 27 + 54 = 108$, and, hence,
\[ [K_2:K] \geq 6\cdot 108= 648 = |E_2|. \]
Since $\Gal(K_2/K)$ is isomorphic to a subgroup of $E_2$,
we must have $\Gal(K_2/K)\cong E_2$.
\end{proof}

We are now prepared to prove part (b) of Theorem~\ref{Thm:intro}, which we
restate here.

\begin{theorem}
\label{Thm: Main}
Let $K$ and $x$ be as in Proposition~\ref{prop:div6n}. Then for any $n\geq 1$,
\[ \Gal\Big( K\big(f^{-n}(x) \big) / K \Big) \cong E_n. \]
\end{theorem}

\begin{proof}
By Lemma~\ref{Lem:GnsubEn}, we know that the Galois group is isomorphic to
a subgroup of $E_n$. We must show that this subgroup is $E_n$ itself.
We proceed by induction on $n$.
For $n=1,2$, we are done by Proposition~\ref{prop:E2}.
Assuming we know the statement (for any such $K$ and $x$) for a particular $n\geq 2$,
we will now show it for $n+1$.

Let
\[ K_n= K\big(f^{-n}(x)\big) \quad\text{and}\quad K_{n+1}= K\big(f^{-(n+1)}(x)\big). \]
Write $f^{-1}(x) = \{u_1,u_2,u_3\}$.
By Proposition~\ref{prop:div6n}, each of the pairs $(K(u_i),u_i)$ satisfies
property~\eqref{Dagger property} for $i=1,2,3$. Thus, our induction hypothesis
says that $\Gal(K_n/K)$ and all three of the Galois groups
\[ \Gal\Big(K\big(f^{-n}(u_i)\big) / K(u_i) \Big),
\qquad\text{for}\; i=1,2,3 \]
are isomorphic to $E_n$.
Pick
\[ y\in f^{-(n-2)}(u_3) \subseteq f^{-(n-1)}(x) . \]
See Figure~\ref{fig:tree3marked}.
\begin{figure}
\includegraphics{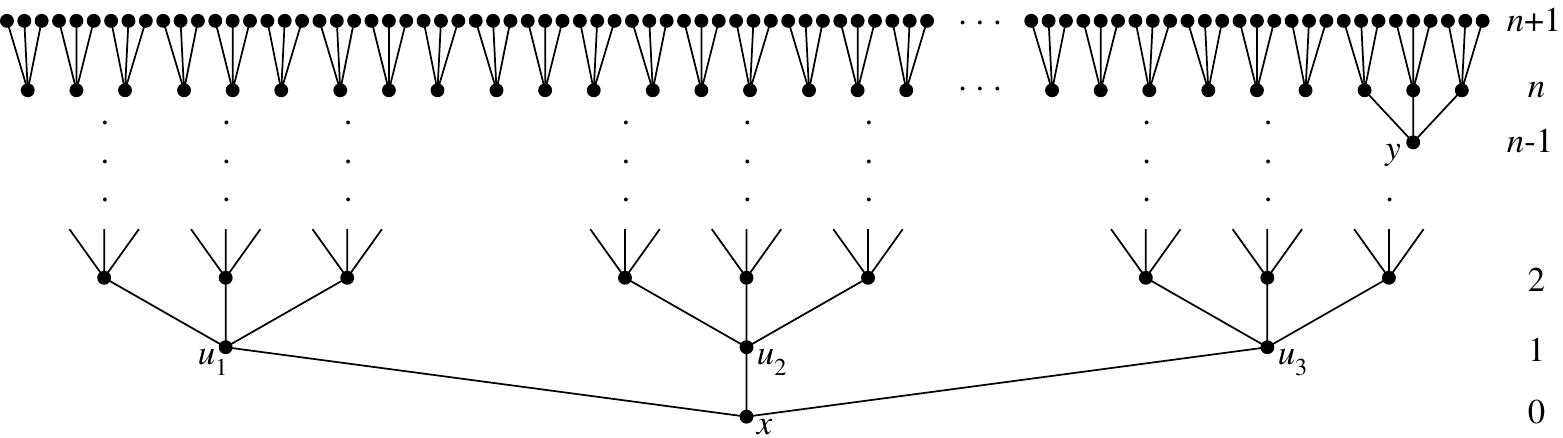}
\caption{The locations of $x$, $u_1$, $u_2$, $u_3$ and $y$ in $T_{n+1}$.}
\label{fig:tree3marked}
\end{figure}
Our main goal, which we will achieve at the end of Step~3 below,
is to construct an element $\lambda\in\Gal(K_{n+1} / K)$ that is the identity on
\[ f^{-(n+1)}(x) \smallsetminus f^{-2}(y) \quad\text{and on}\quad f^{-1}(y),  \]
but which acts as two disjoint 2-cycles on $f^{-2}(y)$.

\medskip

\noindent
\textbf{Step 1.}
Define
\[ H= \Gal\Big(K\big(f^{-n}(u_3)\big) / K(u_3) \Big).\]
By the induction hypothesis, $H \cong E_n$. By
Lemma~\ref{Lem: Special elts} there is some $\sigma_1 \in H$ that is the identity on
\[ f^{-n}(u_3) \smallsetminus f^{-2}(y) \quad\text{and on}\quad f^{-1}(y)  \]
and acts as two 2-cycles on $f^{-2}(y)$. Lift $\sigma_1$ to
\[ \sigma\in \Gal\big(K_{n+1} / K(u_3) \big)\subseteq \Gal(K_{n+1} / K). \]
Thus, $\sigma$ acts as we would like $\lambda$ to act on
$f^{-n}(u_3)$, but we have no idea how it acts on $f^{-n}(u_1)$ and
$f^{-n}(u_2)$.

\medskip

\noindent
\textbf{Step 2.} By Lemma~\ref{Lem: Special elts}, we may pick
$\tau_1\in \Gal(K_n/K)\cong E_n$ that is the identity on
\[ f^{-n}(x) \smallsetminus f^{-1}(y) \]
and acts as a 3-cycle on $f^{-1}(y)$.
Lift $\tau_1$ to $\tau\in \Gal(K_{n+1} / K)$.

Then $\tau\sigma\tau^{-1}\in\Gal(K_{n+1}/K)$ acts as
\begin{itemize}
\item the identity on $f^{-n}(u_3) \smallsetminus f^{-2}(y)$ and $f^{-1}(y)$,
\item two 2-cycles on $f^{-2}(y)$, and
\item the same as $\sigma$ on $f^{-(n-1)}(\{ u_1, u_2\})$,
\end{itemize}
where the two 2-cycles on $f^{-2}(y)$ for $\tau\sigma\tau^{-1}$ do not
occur above the same two elements of $f^{-1}(y)$ as the two 2-cycles for $\sigma$.

Thus, $\tau\sigma\tau^{-1}\sigma^{-1}\in\Gal(K_{n+1}/K)$ acts as
\begin{itemize}
\item the identity on $f^{-n}(u_3) \smallsetminus f^{-2}(y)$ and $f^{-1}(y)$,
\item two 2-cycles and (perhaps) a separate 3-cycle on $f^{-2}(y)$, and
\item the identity on $f^{-(n-1)}(\{ u_1, u_2\})$.
\end{itemize}
Cubing to kill the possible 3-cycle in $f^{-2}(y)$, we see that
\[ \rho = (\tau\sigma\tau^{-1}\sigma^{-1})^3 \in\Gal(K_{n+1}/K)\]
acts as
\begin{itemize}
\item the identity on $f^{-n}(x)$,
\item the identity on $f^{-n}(u_3) \smallsetminus f^{-2}(y)$, and
\item two 2-cycles on $f^{-2}(y)$.
\end{itemize}

\medskip

\noindent
\textbf{Step 3.}
Consider the permutations $\tau,\rho\in\Gal(K_{n+1}/K)$ of Step~2.
Then $\tau\rho\tau^{-1}\rho^{-1}$ acts as
\begin{itemize}
\item the identity on $f^{-n}(x)$,
\item the identity on $f^{-n}(u_3) \smallsetminus f^{-2}(y)$,
\item two 2-cycles and (perhaps) a separate 3-cycle on $f^{-2}(y)$, and
\item for each $v\in f^{-(n-1)}(\{u_1,u_2\})$, an even permutation of $f^{-1}(v)$.
\end{itemize}
The even permutations of the last bullet point above are even permutations in $\fS_3$,
and, hence, each is either the identity or a 3-cycle. Cubing, we see that
\[ \lambda = (\tau\rho\tau^{-1}\rho^{-1})^3 \in\Gal(K_{n+1}/K)\]
acts as the identity on
\[ f^{-(n+1)}(x) \smallsetminus f^{-2}(y) \quad\text{and on}\quad f^{-1}(y),  \]
and it acts as two 2-cycles on $f^{-2}(y)$,
achieving the main goal from the start of the proof.

\medskip

\noindent
\textbf{Step 4.}  Recall $H=\Gal(K(f^{-n}(u_3)) / K(u_3))\cong E_n$.
Pick $w\in f^{-1}(y)$. By Lemma~\ref{Lem: Special elts}, we can pick
$\gamma_1\in H$ that is the identity on
\[ f^{-n}(u_3) \smallsetminus f^{-1}(w) \]
and acts as a $3$-cycle on $f^{-1}(w)$. Lift $\gamma_1$ to $\gamma\in\Gal(K_{n+1}/K)$.

Conjugating the permutation $\lambda$ (of Step~3) by various products
of $\gamma$ and the permutation $\tau$ (of Step~2), we see that
$\Gal(K_{n+1}/K)$ contains each of the $27$ permutations that is the
identity on
\[ Y= \big( f^{-(n+1)}(x) \smallsetminus f^{-2}(y) \big) \cup f^{-1}(y) \]
and acts as two disjoint $2$-cycles on $f^{-2}(y)$.  In addition,
taking products of pairs of such permutations, $\Gal(K_{n+1}/K)$
contains all $27$ permutations that are the identity on $Y$ and
products of disjoint $3$-cycles on $f^{-2}(y)$.  Still taking products
of pairs, $\Gal(K_{n+1}/K)$ also contains all $54$ permutations that
are the identity on $Y$ and the product of a disjoint $3$-cycle and
two $2$-cycles on $f^{-2}(y)$.  Together, then, $\Gal(K_{n+1}/K)$
contains a subgroup $H_y$ that acts trivially on $Y$, with $|H_y|=108
= 2^2 \cdot 3^3$.

Since $f^{n-1}(z) - x$ is irreducible over $K$, for each root $y' \in f^{-(n-1)}(x)$ there is some
\[ \delta_1\in\Gal(K_{n-1}/K)\cong E_{n-1}\]
with $\delta_1(y')=y$.
Lift $\delta_1$ to $\delta\in\Gal(K_{n+1}/K)$. Then,
\[
H_{y'} = \delta^{-1} H_y \delta
\]
is a 108-element subgroup of $\Gal(K_{n+1}/K)$ that acts trivially on
\[
Y' = \big( f^{-(n+1)}(x) \smallsetminus f^{-2}(y') \big) \cup f^{-1}(y').
\]
There are $3^{n-1}$ choices for $y'$, and any two of the resulting subgroups $H_{y'}$
act nontrivially on disjoint portions of the preimage tree.
In addition, they all act trivially on $f^{-n}(x)$ and, hence, trivially on $K_n$.
Thus, the product of all of them forms a subgroup $B\subseteq\Gal(K_{n+1}/K_n)$
of order
\[
\big( 2^2 \cdot 3^3 \big)^{3^{(n-1)}}
= \frac{|E_{n+1}|}{|E_n|},
\]
by Proposition~\ref{Prop: En order}.
Hence,
\[ \big|\Gal(K_{n+1}/K)\big| = \big|\Gal(K_{n+1}/K_n)\big| \cdot \big|\Gal(K_n/K)\big|
\geq |B| \cdot |E_n| = |E_{n+1}| .\]
Since $\Gal(K_{n+1}/K)$ is isomorphic to a subgroup of the finite group $E_{n+1}$,
we must therefore have $\Gal(K_{n+1}/K)\cong E_{n+1}$.
\end{proof}


\section{The geometric representation}
\label{Sec: Geometric case}

Let $L$ be a number field, and consider the rational function field $K=L(t)$.
Since $t\in L(t) \smallsetminus\{0,1\}$, Lemma~\ref{Lem:GnsubEn} states
that the Galois group of $L(f^{-n}(t))$ over $L(t)$ is a subgroup of $E_n$.
In fact, we have the following much stronger statement.

\begin{proposition}
Let $L$ be a number field and $t$ a transcendental element over $L$. Then
\[
\Gal\Big(L\big(f^{-n}(t)\big)/L(t)\Big)\cong E_n.
\]
\end{proposition}

\begin{proof}
Let $G_n=\Gal(L(f^{-n}(t))/L(t))$.
We can choose $x\in L$ such that the pair $(L,x)$ satisfies
property~\eqref{Dagger property} of Section~\ref{Sec: main theorem}.
By Theorem~\ref{Thm: Main}, we have
\[
\Gal\Big(L\big(f^{-n}(x)\big)/L(x)\Big) =
\Gal\Big(L\big(f^{-n}(x)\big)/L\Big) \cong E_n.
\]
Therefore, the specialization lemma of
\cite[Lem.~2.4]{Odoni_Galois_Iterates_1985} implies that $G_n$ has a
subgroup isomorphic to $E_n$.  On the other hand, applying
Lemma~\ref{Lem:GnsubEn} to $G_n$ shows that $G_n$ is isomorphic to a
subgroup of $E_n$.  Since $E_n$ is finite, $G_n$ must be isomorphic to
$E_n$.
\end{proof}

\begin{corollary}
$\Gal\Big(\bar{\QQ}\big(f^{-n}(t)\big)/\bar{\QQ}(t)\Big)\cong E_n$.
\end{corollary}
	
\begin{proof}
Since the previous proposition holds for any number field, $\QQ$ must
be algebraically closed in $\QQ(f^{-n}(t))$. Hence,
\[
\Gal\Big(\bar{\QQ}\big(f^{-n}(t)\big)/\bar{\QQ}(t)\Big)
\cong \Gal\Big(\QQ\big(f^{-n}(t)\big)/\QQ(t)\Big)\cong E_n. \qedhere
\]
\end{proof}


\section{Counting elements that fix leaves of $T_n$}
\label{Sec: Counting fixing elements}

Write $E_{n,\fix}$ for the set of elements of $E_n$ that fix at least one leaf
of $T_n$. We have already seen that $E_\infty = \varprojlim E_n$ is the geometric
monodromy group of the PCF polynomial $f(z) = -2z^3 + 3z^2$. Using this fact,
one could apply \cite[Thm.~1.1]{Jones_FFF_monodromy} to show that the ratio
$|E_{n,\fix}|/|E_n|$ tends to zero with $n$. And while this would be sufficient
for the arithmetic applications in the next section, we are able to obtain a
more refined statement by working directly with the group structure of $E_n$:

\begin{theorem}
\label{Thm: Fixing elements}
The proportion of elements of $E_n$ that fix a leaf of $T_n$ is
\[
\frac{|E_{n,\fix}|}{|E_n|}
= \frac{2}{n}\left(1 + O\left(\frac{\log n}{n}\right)\right) \quad \text{as $n \to \infty$}.
\]
\end{theorem}


\begin{remark}
The proportion of elements of $\Aut(T_n)$ that fix a leaf of $T_n$
obeys the same asymptotic as for $E_n$
\cite[\S4]{Odoni_Galois_Iterates_1985}. By way of contrast, consider
$H_n \cong [C_3]^n$, the Sylow 3-subgroup of $E_n$ from
Proposition~\ref{Prop: sylow}. The proportion of elements of $H_n$
that fix a leaf of $T_n$ is half that of $E_n$: $\frac{1}{n}\left(1 +
O\left(\frac{\log n}{n}\right)\right)$.
\end{remark}

We begin by finding a recursive formula for $|E_{n,\fix}|$ in terms of
certain auxiliary quantities. For $n \geq 1$ and $i \in \{1,2,3\}$, we
define the following:
\begin{align*}
  A_{n,i} &= \Big| \big\{ s \in E_n \, \big| \,
  s \text{ acts as an element of order $i$ on $T_1$} \big\} \Big| \\
  A_{n,i}' &= \Big| \big\{ s \in E_n \, | \,
  s \text{ acts as an element of order $i$ on $T_1$ and fixes a leaf of $T_n$} \big\} \Big|.
\end{align*}
For example, if $n = 1$, we have
\begin{align*}
A_{1,1} = 1 \quad & \quad A_{1,1}' = 1 \\
A_{1,2} = 3 \quad & \quad A_{1,2}' = 3 \\
A_{1,3} = 2 \quad & \quad A_{1,3}' = 0.
\end{align*}

For any $n\geq 1$, note that $A_{n,3}'=0$, because an element $s$
that permutes the leaves of $T_1$ by a $3$-cycle cannot fix a leaf of $T_n$.
It follows that
\[
|E_{n,\fix}| = A_{n,1}' + A_{n,2}'  \quad \text{and} \quad
|E_n| = A_{n,1} + A_{n,2} + A_{n,3}.
\]

\begin{lemma}
\label{Lem: A1}
For $n \geq 1$, we have
\[
A_{n,2} = 3A_{n,1}, \qquad A_{n,3} = 2 A_{n,1}, \qquad |E_n| = 6 A_{n,1},
\qquad \text{and} \qquad |E_{n+1}| = 3 |E_n|^3.
\]
\end{lemma}

\begin{proof}
  The restriction homomorphism $\pi \colon E_n \to E_1 \cong \fS_3$ is
  onto since $\pi\big( (1,1,1), (123) \big) = (123)$ and $\pi(\tau_n) =
  (12)$, where $\tau_n$ was defined in~\eqref{Eq: tau}. The first three
  equalities follow from the fact that $A_{n,1} = |\ker(\pi)|$. For the final
  equality, apply Proposition~\ref{Prop: En order}.
\end{proof}





\begin{lemma}
\label{Lem: A1 prime}
For $n \geq 1$, we have
\[
A_{n+1,1}' = 54A_{n,1}^2\left(A_{n,1}' + A_{n,2}'\right) - 9 A_{n,1}\left(A_{n,1}' + A_{n,2}'\right)^2
+ 3A_{n,1}'\left(A_{n,2}'\right)^2 + \left(A_{n,1}'\right)^3.
\]
\end{lemma}

\begin{proof}
%
Let $s \in E_{n+1}$ be an element that acts as the identity on $T_1$. Then the restriction of $s$ to $T_2$ is of the form $\big((a_1,a_2,a_3),1\big)$ for some $a_1,a_2,a_3\in \Aut(T_1)$.
By Lemma~\ref{Lem: Sign}, the fact that this element lies in $E_2$ means $1 = \prod \sgn(a_i)$.
So among the $a_i$'s, there are either zero 2-cycles or exactly two 2-cycles. We treat these cases separately.

\medskip

\noindent \textbf{Case 1: Zero 2-cycles.} As an element of $E_{n+1}$, we have
$s = \big((\tilde a_1, \tilde a_2, \tilde a_3), 1 \big)$, where each $\tilde a_i \in E_n$ restricts to either the identity or a 3-cycle on $T_1$. The total number of elements $\tilde a_i$ of this shape is
$A_{n,1} + A_{n,3}$, while the number that \text{do not} fix a leaf of $T_n$ is
$\big(A_{n,1} + A_{n,3} - A_{n,1}'\big)$.
Thus, the number of elements of $E_{n+1,\fix}$ that act as the identity on $T_1$
but with no 2-cycle on $T_2$ is
\begin{equation}
\label{Eq: first part}
\left(A_{n,1} + A_{n,3}\right)^3 - \left(A_{n,1} + A_{n,3} - A_{n,1}'\right)^3
= 27 A_{n,1}^2 A_{n,1}' - 9 A_{n,1} \left(A_{n,1}'\right)^2 + \left(A_{n,1}'\right)^3,
\end{equation}
where we applied Lemma~\ref{Lem: A1} when we expanded the two expressions.
\medskip

\noindent \textbf{Case 2: Two 2-cycles.} As an element of $E_{n+1}$,
we have $s = \big((\tilde a_1, \tilde a_2, \tilde a_3), 1 \big)$,
where two of the $\tilde a_i \in E_n$ restrict to 2-cycles on $T_1$,
and the remaining $\tilde a_i$ restricts to the identity or a 3-cycle.
There are 3 choices for the index $i_0$ such that $\tilde a_{i_0}$ is
the identity or a 3-cycle.  For a given choice of $i_0$, there are
\[ (A_{n,1}' + A_{n,3}')A_{n,2}^2 = A_{n,1}'A_{n,2}^2 = 9 A_{n,1}^2 A_{n,1}' \]
choices of triples $(\tilde a_1,\tilde a_2,\tilde a_3)$ such that
$s$ fixes a leaf of the copy of $T_n$ above the $i_0$-leaf of $T_1$.

For $s$ \emph{not} to fix a leaf of the $T_n$ above the $i_0$-leaf,
at least one of the other two
$\tilde a_i$'s (each of which acts as a 2-cycle on $T_1$) must fix a leaf of $T_n$.
By inclusion-exclusion, the
number of choices for this pair of $\tilde a_i$'s is
\[ 2A_{n,2}A_{n,2}' - \left(A_{n,2}'\right)^2 = \left( 6 A_{n,1} - A_{n,2}' \right) A_{n,2}' . \]
For $i=i_0$, the element $\tilde a_i\in E_n$ acts as the identity or as a 3-cycle on $T_1$
but fixes no leaf of $T_n$. The number of such elements of $E_n$ is
\[  A_{n,1} + A_{n,3} - A_{n,1}' = 3 A_{n,1} - A_{n,1}'.  \]
(Again, we have applied Lemma~\ref{Lem: A1} in all three displayed equations above.)

Thus, the number of elements $s\in E_{n+1,\fix}$ that act as the identity on
$T_1$ and as two 2-cycles on the leaves of $T_2$ is
\begin{equation}
\label{Eq: second part}
3\Big[ 9A_{n,1}^2A_{n,1}' +
\left(3A_{n,1} - A_{n,1}'\right)\left(6A_{n,1} - A_{n,2}'\right) A_{n,2}' \Big].
\end{equation}

Adding \eqref{Eq: first part} and \eqref{Eq: second part} and expanding yields
the desired result.
\end{proof}

Using the same counting technique as in the preceding proof, one obtains:

\begin{lemma}
\label{Lem: A2 prime}
For $n \geq 1$, we have
\[
A_{n+1,2}' = 54 A_{n,1}^2 \left(A_{n,1}' + A_{n,2}'\right).
\]
\end{lemma}

\begin{lemma}
\label{Lem: phi}
Let $\phi(t) = t - \frac{1}{2}t^2 + \frac{1}{3}t^3$.  Then $\phi$ is
increasing on $(0,1)$, and  
\[
\phi^n(1) =
\frac{2}{n}\left(1 + O\left(\frac{\log n}{n}\right)\right) \text{ as } n \to \infty.
\]
\end{lemma}

\begin{proof}
  The first statement is evident from looking at the derivative. For
  the second, we set
  \[
  \psi(z) = \frac{1}{\phi(z^{-1})} = z + \frac{1}{2} - \frac{z+2}{2(6z^2 - 3z + 2)}.
  \]
  Let $R(z) = \frac{z+2}{2(6z^2 - 3z + 2)}$ be the final term. Then, by
  induction, we have
  \[
  \psi^n(z) = \frac{1}{\phi^n(z^{-1})} = z + \frac{n}{2}
  - \sum_{i=0}^{n-1} R\left(\psi^i(z)\right).
  \]
  Since $\phi^n(1) = 1 / \psi^n(1)$, to complete the proof it
  suffices to show that $\sum_{i=0}^{n-1} R\left(\psi^i(1)\right) =
  O(\log n)$.

By elementary algebra, one verifies that $R(z) \leq \frac{1}{3z}$ for all
$z>0$. Now we show, by induction, that $\psi^n(1) \geq (n+5)/5$ for $n\geq
1$. Both sides equal $6/5$ for $n=1$. Given $n\geq 1$ for which the inequality
holds, we find $R(\psi^n(1)) \leq \frac{1}{3\psi^n(1)} \leq \frac{5}{3n+15}$,
and hence
\begin{align*}
\psi^{n+1}(1) &= \psi^n(1) + \frac{1}{2} - R(\psi^n(1))
\geq \frac{n+5}{5} + \frac{1}{2} - \frac{5}{3n+15} \\
& = \frac{n+6}{5} + \frac{9n-5}{30(n+5)} > \frac{n+6}{5},
\end{align*}
completing the induction.
  
  Using our inequalities for $R(z)$ and $\psi^n(1)$, we conclude that
  \[
  0 \leq \sum_{i=0}^{n-1} R\left(\psi^i(1)\right)
  < \sum_{i=0}^{n-1} \frac{1}{\psi^i(1)} < 5 \sum_{i=0}^{n-1} \frac{1}{i+5}
  = O(\log n). \qedhere
  \]
\end{proof}

One can apply the technique in the previous proof to obtain the
following similar result:

\begin{lemma}
\label{Lem: rho}
Let $\rho(t) = t - \frac{1}{2}t^2$.  Then $\rho$ is increasing on
$(0,1)$, and 
\[
\rho^n(2/3) =
\frac{2}{n}\left(1 + O\left(\frac{\log n}{n}\right)\right) \text{ as } n \to \infty.
\]
\end{lemma}

\begin{proof}[Proof of Theorem~\ref{Thm: Fixing elements}]
By adding the terms $\left(A_{n,2}'\right)^3$ and
$3A_{n,2}'\left(A_{n,1}'\right)^2$ to the formula in Lemma~\ref{Lem: A1 prime},
we obtain the estimate
\[
A_{n+1,1}' \leq 54A_{n,1}^2\left(A_{n,1}' + A_{n,2}'\right) - 9 A_{n,1}\left(A_{n,1}' + A_{n,2}'\right)^2
+ \left(A_{n,1}' + A_{n,2}'\right)^3.
\]
Adding this to the formula in Lemma~\ref{Lem: A2 prime}, we find that
\begin{align*}
|E_{n+1,\fix}| = A_{n+1,1}' + A_{n+1,2}' &\leq
108A_{n,1}^2\left(A_{n,1}' + A_{n,2}'\right) - 9 A_{n,1}\left(A_{n,1}' + A_{n,2}'\right)^2
+ \left(A_{n,1}' + A_{n,2}'\right)^3 \\
&= 3\left(6A_{n,1}\right)^2 |E_{n,\fix}| - \frac{3}{2} \left(6A_{n,1}\right)|E_{n,\fix}|^2 + |E_{n,\fix}|^3 \\
&= 3 |E_n|^2 \cdot |E_{n,\fix}| - \frac{3}{2} |E_n| \cdot |E_{n,\fix}|^2 + |E_{n,\fix}|^3,
\end{align*}
where we have used Lemma~\ref{Lem: A1} to write $6A_{n,1} = |E_n|$. Dividing by $|E_{n+1}| = 3|E_n|^3$, we find that
\begin{equation}
\label{Eq: En upper estimate}
\frac{|E_{n+1,\fix}|}{|E_{n+1}|} \leq \frac{|E_{n,\fix}|}{|E_n|}
- \frac{1}{2}\left(\frac{|E_{n,\fix}|}{|E_n|}\right)^2
+ \frac{1}{3}\left(\frac{|E_{n,\fix}|}{|E_n|}\right)^3
= \phi\left(\frac{|E_{n,\fix}|}{|E_n|}\right),
\end{equation}
where $\phi$ is the polynomial from Lemma~\ref{Lem: phi}.

Similarly, by discarding the final two terms of the formula in Lemma~\ref{Lem:
  A1 prime} and adding the formula from Lemma~\ref{Lem: A2 prime}, we obtain the
estimate
\begin{align*}
  |E_{n+1,\fix}| = A_{n+1,1}' + A_{n+1,2}' &\geq
  108A_{n,1}^2\left(A_{n,1}' + A_{n,2}'\right) - 9 A_{n,1}\left(A_{n,1}' + A_{n,2}'\right)^2 \\
  &= 3 |E_n|^2 \cdot |E_{n,\fix}| - \frac{3}{2} |E_n| \cdot |E_{n,\fix}|^2. 
\end{align*}
As above, this yields
\begin{equation}
\label{Eq: En lower estimate}
\frac{|E_{n+1,\fix}|}{|E_{n+1}|} \geq \frac{|E_{n,\fix}|}{|E_n|}
- \frac{1}{2}\left(\frac{|E_{n,\fix}|}{|E_n|}\right)^2 
= \rho\left(\frac{|E_{n,\fix}|}{|E_n|}\right),
\end{equation}
where $\rho$ is the polynomial from Lemma~\ref{Lem: rho}.

Set $x_n = \frac{|E_{n,\fix}|}{|E_n|} \in [0,1]$. Equations~\eqref{Eq: En upper
  estimate} and \eqref{Eq: En lower estimate} show that $\rho(x_n) \leq x_{n+1}
\leq \phi(x_n)$.  As $\rho$ and $\phi$ are increasing on $(0,1)$, we find that
\[
  \rho^n(2/3) = \rho^n(x_1) \leq \rho^{n-1}(x_2) \leq \cdots
  \leq \rho(x_n) \leq x_{n+1}
\]
and
\[
x_{n+1} \leq \phi(x_n) \leq \phi^2(x_{n-1}) \leq
\cdots \leq \phi^n(x_1) \leq \phi^n(1).
\]
By Lemmas~\ref{Lem: phi} and~\ref{Lem: rho}, the first and last
quantities have the same asymptotic value, namely $\frac{2}{n}\left(1
+ O\left(\frac{\log n}{n}\right)\right)$, and hence, so does
$x_{n+1}$. The proof is complete since this asymptotic is unchanged
upon replacing $n$ with $n-1$.
\end{proof}


\section{Arithmetic applications}
\label{Sec: Applications}

We now prove our applications on density of prime divisors in orbits
and Newton's method. If $K$ is a number field and $\fP$ is a prime
ideal of the ring of integers of $K$ with residue field $k(\fP)$,
there is a surjective reduction map $K \to k(\fP) \cup \{\infty\}$. We
write $x \equiv y \pmod \fP$ whenever $x,y \in K$ have the same
reduction.

\begin{proposition}
\label{Prop:ZeroDense}
Let $K$ be a number field and $x \in K$ an element such that $(K,x)$
satisfies property \eqref{Dagger property} of Section~\ref{Sec: main theorem}.
Choose $y_0 \in K \smallsetminus \{x\}$ and define a sequence $(y_i)_{i\geq 0}
\subseteq K$ by $y_i = f^i(y_0)$.  Then the set of prime ideals $\fP$
of $K$ such that
\[
y_i \equiv x \! \! \! \! \pmod \fP \quad \text{ for some $i \geq 0$}
\]
has natural density zero.
\end{proposition}

\begin{proof}
Note that for all $i\geq 0$, we have $y_i \neq x$. This inequality holds for $i =
0$ by hypothesis. Furthermore, if it failed for some $i > 0$, $y_0$ would be a
$K$-rational root of $f^i(z) - x$, which is absurd
since this polynomial is irreducible over $K$ by Lemma~\ref{Lem: eisenstein}.

For each $n\geq 1$, define $S_n$ to be the set of prime ideals $\fP$ of $K$ such that
\begin{itemize}
\item $x$ is not integral at $\fP$, or
\item $y_i \equiv x \! \pmod \fP$ for some $0\leq i \leq n-1$.
\end{itemize}
Then $S_n$ is finite.

Let $n\geq 1$. Let $\fP \not\in S_n$ be a prime ideal of the ring of
integers of $K$ such that $y_i \equiv x \pmod \fP$ for some $i \geq
n$. Then $f^n(y_{i-n}) \equiv x \pmod \fP$, and therefore the
polynomial $f^n(z) - x$ has a $k(\fP)$-rational root.  Write
$\delta(S)$ for the natural density of a set of prime ideals $S$ (if
it exists). For each $n\geq 1$, the Chebotarev Density Theorem and the
finiteness of $S_n$ yield
\begin{align*}
\delta\big( \big\{ \fP \, \big| \, & y_i \equiv x \! \! \! \! \pmod \fP \text{ for some $i \geq 0$} \big\}\big)  \\
&\leq \delta \big( \big\{ \fP \, \big| \, \fP \not\in S_n, f^n(z) - x \text{ has a $k(\fP)$-rational root} \big\}\big) \\
& = \frac{\Big|\Big\{s \in \Gal\big(K(f^{-n}(x))/K\big) \, \Big| \, s \text{ fixes some root of $f^n(z) - x$}\Big\}\Big|}
{\big|\Gal\big(K(f^{-n}(x))/K\big)\big|}
= \frac{| E_{n,\fix}|}{|E_n|}.
\end{align*}
The final equality uses Theorem~\ref{Thm: Main} to identify the Galois
group of $f^n(z)-x$ over $K$ with $E_n$.  By Theorem~\ref{Thm: Fixing
  elements}, this last quantity tends to zero as $n\to\infty$.
\end{proof}

Recall that the critical points for $f$ are $0,1,\infty$, and that they are
all fixed by $f$.

\begin{corollary}
\label{Cor:ZeroDense}
Let $K$ be a number field for which there exist
unramified primes above $2$ and above $3$.
Let $y_0 \in K \smallsetminus \{0,1,3/2, -1/2\}$,
and define a sequence $(y_i)_{i\geq 0} \subseteq K$ by $y_i = f^i(y_0)$.
Then the set of prime ideals $\fP$ of $K$ such that
\[
y_i \equiv 0 \text{ or } 1 \! \! \! \! \pmod \fP \quad \text{ for some $i \geq 0$}
\]
has natural density zero. In particular, the set of prime divisors of the sequence
$(y_i)_{i\geq 0}$ has natural density zero.
\end{corollary}


\begin{proof}
We begin by showing that the set of primes $\fP$ such that
$y_i \equiv 0 \pmod \fP$ for some $i \geq 0$ has natural density zero.
Let $S$ be the set of primes $\fP$ of $K$ for which
\begin{itemize}
\item $\fP$ lies above $2$, or
\item $y_0 \equiv 0 \pmod \fP$.
\end{itemize}
Notice that since $y_0 \neq 0$, $S$ is a finite set, and we may safely
ignore primes in $S$ for the remainder of the proof.

Suppose now that $\fP \not\in S$ is such that $y_i \equiv 0 \pmod \fP$
for some $i \geq 1$.  We may assume without loss that $i$ is minimal
with this property. Then
\[
y_i = f(y_{i-1}) = -2y_{i-1}^3 + 3y_{i-1}^2 = - 2 y_{i-1}^2\left(y_{i-1} - \frac{3}{2}\right),
\]
which implies that $y_{i-1} \equiv \frac{3}{2} \pmod \fP$. We claim
that $y_{i-1} \neq \frac{3}{2}$. This is true by hypothesis if $i =
1$. If it were to fail for some $i > 1$, then the polynomial
$f^{i-1}(z) - 3/2$ would have $y_0$ as a $K$-rational root. But our
hypothesis on $K$ implies that $(K,3/2)$ satisfies
property~\eqref{Dagger property}, and so we have a contradiction to
Lemma~\ref{Lem: eisenstein}. It follows that
\[
\delta\Big( \big\{ \fP \, \big| \, y_i \equiv 0 \!\!\!\! \pmod \fP \text{ for some } i \geq 0 \big\} \Big)
= \delta\bigg( \Big\{ \fP \, \Big| \, y_i \equiv \frac{3}{2} \!\!\!\!\pmod \fP
\text{ for some } i \geq 1 \Big\} \bigg).
\]
Since $(K,3/2)$ satisfies property~\eqref{Dagger property}, the
density on the right is zero by Proposition~\ref{Prop:ZeroDense}.

Now we show that the density of primes $\fP$ such that $y_i \equiv 1
\pmod \fP$ for some $i \geq 0$ also has natural density zero. Define
$w_i = 1 - y_i$ for $i \geq 0$. As $y_0 \not\in \{1,-1/2\}$, we see
that $w_0 \not\in \{0,3/2\}$.  Moreover, because $1 - f(z) = f(1-z)$, we
find that
\[
f(w_i) = f(1-y_i) = 1 - f(y_i) = 1 - y_{i+1} = w_{i+1}.
\]
Thus, we may apply the first part of the proof to the sequence $(w_i)_{i\geq 0}$ to deduce that
\[
\delta\big(\{ \fP \,\big| \, w_i \equiv 0 \!\!\!\! \pmod \fP \text{ for some $i \geq 0$}\} \big)= 0.
\]
Since $w_i \equiv 0 \! \pmod \fP \ \Leftrightarrow \ y_i \equiv 1 \! \pmod \fP$, we are done.
\end{proof}

We now prove a special case of the Faber-Voloch conjecture. Recall
that the \emph{Newton map} associated to a polynomial $g(z) \in K[z]$
is the rational function
\[
N_g(z) = z - \frac{g(z)}{g'(z)}.
\]
The simple roots of $g$ are critical fixed points of $N_g$. Hence, for
any completion $K_v$ of $K$, the roots of $g$ are super-attracting
fixed points of the map $N_g$, viewed as a dynamical system acting on
$\PP^1(K_v)$.

\begin{corollary}
Let $K$ be a number field for which there exist
unramified primes above $2$ and above $3$.
Let $g(z) = z^3 - z$. Choose $y_0 \in K$ such that the Newton iteration sequence
$y_i = N_g^i(y_0)$ does not encounter a root of $g$.
Then the set of primes $\fP$ of $K$ for which the Newton sequence $(y_i)_{i\geq 0}$
converges in $K_\fP$ to a root of $g$ has natural density zero.
\end{corollary}

\begin{proof}
The Newton map for $g$ is
\[
N_g(z) = \frac{2z^2}{3z^2 - 1}.
\]
Let $\eta(z)=1/(1-2z)$. Then
\begin{equation}
\label{Eq:NgConj}
\eta^{-1} \circ N_g \circ \eta(z) = -2z^3 + 3z^2 = f(z).
\end{equation}

For each $i\geq 0$, define $w_i = \eta^{-1}(y_i)$. Then it is
immediate from~\eqref{Eq:NgConj} that $w_{i+1} = f(w_i)$. Moreover, if
we had $w_0 \in \{0,1,3/2,-1/2, \infty\}$, then the sequence
$(w_i)_{i\geq 0}$ would encounter a fixed point of $f$, in which
case $(y_i)_{i\geq 0}$ would encounter a fixed point of $N_g$ and,
hence, a root of $g$, contradicting our hypotheses. Thus, $w_0 \not\in
\{0,1,3/2,-1/2, \infty\}$.  Corollary~\ref{Cor:ZeroDense} therefore
shows that the set of prime ideals $\fP$ for which $w_i \equiv 0
\text{ or } 1 \!\pmod \fP$ has density zero.

On the other hand, the proof of the main theorem of Faber-Voloch
\cite{Faber_Voloch_Bordeaux_2011} shows that for all but finitely many
prime ideals $\fP$ of $K$, the sequence $(y_i)_{i \geq 0}$ converges
in $K_\fP$ to a root of $g$ if and only if $g(y_i) \equiv 0 \! \pmod
\fP$ for some $i \geq 0$.  Factoring $g$, this condition is equivalent
to saying that $y_i \equiv 0, \pm 1 \!\pmod \fP$ for some $i \geq 0$,
which in turn is equivalent to saying that $w_i \equiv 0,1, \text{ or
} \infty \!\pmod \fP$.  (Here, $w\equiv \infty \!\pmod \fP$ means $w$
is not integral at $\fP$.)  Clearly, the set of primes $\fP$ for which
$w_i \equiv \infty \!\pmod \fP$ is zero, since $f$ is a polynomial,
and so the proof is complete.
\end{proof}

\noindent
\textbf{Acknowledgments:} This project began at a workshop on ``The
Galois theory of orbits in arithmetic dynamics'' at the American
Institute of Mathematics in May 2016. We would like to thank AIM for
its generous support and hospitality, Rafe Jones for his early
encouragement to pursue this line of thought, and Clay Petsche for
several early discussions. We also thank the anonymous referees for
pointing out several opportunities for improvement. The first author
gratefully acknowledges the support of NSF grant DMS-1501766. The
third author gratefully acknowledges the support of NSF grant
DMS-1415294.


\bibliographystyle{plain}
\bibliography{BFHJY}

\end{document}